\documentclass[a4paper,12pt]{amsart}

\usepackage{amsmath,amssymb}
\usepackage{ifthen}
\usepackage{graphicx}
\usepackage[T1]{fontenc}
\usepackage[utf8]{inputenc}
\usepackage[usenames,dvipsnames]{color}
\usepackage[english]{babel}
\usepackage{fancyhdr}
\usepackage{fancybox}
\usepackage{tikz}
\usepackage{cite}
\usepackage[backref=page, colorlinks=true, linkcolor=blue, citecolor=red, urlcolor=blue]{hyperref}

\renewcommand*{\backref}[1]{}
\renewcommand*{\backrefalt}[4]{%
	\ifcase #1 %
	(Not cited.)%
	\or
	[p.#2].%
	\else
	[p.#2].%
	\fi
}

\nonstopmode 
\numberwithin{equation}{section}
\theoremstyle{plain}

\newtheorem{conj}{Conjecture}

\theoremstyle{definition}
\newtheorem{defi}{Definition}[section]

\newtheorem{cor}{Corollary}[section]
\newtheorem{thm}{Theorem}[section]

\newtheorem{lem}{Lemma}[section]
\newtheorem{prop}{Proposition}[section]
\newtheorem{prob}{Problem}
\newtheorem{rem}{Remark}[section]

\theoremstyle{plain}

\newtheorem*{lemA}{Lemma A}
\newtheorem*{lemB}{Lemma B}

\newcounter{minutes}
\divide\time by 60
\newcounter{hours}
\multiply\time by 60
\addtocounter{minutes}{-\time}

\newcounter {own}
\def\theown {\thesection.\arabic{own}}

\newenvironment{pf}[1][]{%
	\vskip 3mm
	\noindent
	\ifthenelse{\equal{#1}{}}%
	{{\slshape Proof. }}%
	{{\slshape #1.} }%
}%
{\qed\bigskip}

\newcounter{alphabet}

\def\be{\begin{equation}}
	\def\ee{\end{equation}}

\newcommand{\bee}{\begin{enumerate}}
	\newcommand{\eee}{\end{enumerate}}

\newcommand{\blem}{\begin{lem}}
	\newcommand{\elem}{\end{lem}}
\newcommand{\bthm}{\begin{thm}}
	\newcommand{\ethm}{\end{thm}}
\newcommand{\bcor}{\begin{cor}}
	\newcommand{\ecor}{\end{cor}}
\newcommand{\beg}{\begin{examp}}
	\newcommand{\eeg}{\end{examp}}
\newcommand{\begs}{\begin{examples}}
	\newcommand{\eegs}{\end{examples}}
\newcommand{\bdefn}{\begin{defn}}
	\newcommand{\edefn}{\edefn}
	\newcommand{\bprob}{\begin{prob}}
		\newcommand{\eprob}{\end{prob}}
	\newcommand{\bei}{\begin{itemize}}
		\newcommand{\eei}{\end{itemize}}
	\newcommand{\bcon}{\begin{conj}}
		\newcommand{\econ}{\end{conj}}
	\newcommand{\bprop}{\begin{prop}}
		\newcommand{\eprop}{\eprop}
		\newcommand{\br}{\begin{rem}}
			\newcommand{\er}{\end{rem}}
		\newcommand{\bpf}{\begin{pf}}
			\newcommand{\epf}{\end{pf}}
		\newcommand{\ba}{\begin{array}}
			\newcommand{\ea}{\end{array}}
		\newcommand{\beq}{\begin{eqnarray}}
			\newcommand{\beqq}{\begin{eqnarray*}}
				\newcommand{\eeq}{\end{eqnarray}}
			\newcommand{\eeqq}{\end{eqnarray*}}

\begin{document}

\title{Growth, Distortion, and Schwarzian Norm Estimates for Exponentially Convex Functions}

\author{Molla Basir Ahamed$^*$}
\address{Molla Basir Ahamed, Department of Mathematics, Jadavpur University, Kolkata-700032, West Bengal, India.}
\email{mbahamed.math@jadavpuruniversity.in}

\author{Rajesh Hossain}
\address{Rajesh Hossain, Department of Mathematics, Jadavpur University, Kolkata-700032, West Bengal, India.}
\email{rajesh1998hossain@gmail.com}

\subjclass[{AMS} Subject Classification:]{Primary: 30C45; Secondary: 30C55}
\def\thefootnote{}
\footnotetext{ {\tiny File:~\jobname.tex,
		printed: \number\year-\number\month-\number\day,
		\thehours.\ifnum\theminutes<10{0}\fi\theminutes }
} \makeatletter\def\thefootnote{\@arabic\c@footnote}\makeatother
\begin{abstract} 
	In this paper, we investigate the growth, distortion, pre-Schwarzian and Schwarzian norms of functions in the exponentially convex class \(\mathcal C_{e^\lambda}\), \(0<\lambda\le\pi/2\), defined by \(1+zf''(z)/f'(z)\prec e^{\lambda z}\). By representing the associated Schwarz function explicitly, we derive parameter-dependent estimates for \(f\), \(f'\), and the pre-Schwarzian derivative. We further obtain Schwarzian norm estimates under both the general normalization and the additional condition \(f''(0)=0\). The corresponding extremal problems are analyzed through suitable Schwarz functions, and the dependence of the resulting bounds on the exponential parameter is made explicit.
\end{abstract}
\maketitle
\pagestyle{myheadings}
\markboth{M. B. Ahamed and R. Hossain}{Growth, Distortion, and Norm Estimates for a Class of Exponentially Convex Functions}
\section{\bf Introduction}
Let $\mathcal{H}$ be the class of all analytic functions in the open unit disk $\mathbb{D} := \{z \in \mathbb{C} : |z| < 1\}$ and let $\mathcal{A}$ be the subclass of $\mathcal{H}$ with $f(0) = 0 = f'(0) - 1$. Thus any $f \in \mathcal{A}$ has the following Taylor series form
\begin{equation*}
	f(z) = z + \sum_{n=2}^{\infty} a_n z^n.
\end{equation*}
Further, let $\mathcal{S}$ be the subclass of $\mathcal{A}$ consisting of functions that are univalent (that is, one-to-one) in $\mathbb{D}$. A function $f \in \mathcal{A}$ is called starlike if $f(\mathbb{D})$ is a starlike domain with respect to the origin, \textit{i.e.,} for $w_0 \in f(\mathbb{D})$, the line segment joining $0$ and $w_0$ lies entirely in $f(\mathbb{D})$. The set of all starlike functions in $\mathcal{S}$ is denoted by $\mathcal{S}^*$. It is well-known that a function $f \in \mathcal{A}$ is in $\mathcal{S}^*$ if, and only if, ${\rm Re}\left({zf'(z)}/{f(z)}\right) > 0$ for $z \in \mathbb{D}$. Similarly, a function $f \in \mathcal{A}$ is called convex if $f(\mathbb{D})$ is a convex domain, \textit{i.e.,} $f(\mathbb{D})$ is a convex domain with respect to each of its points. The set of all convex functions in $\mathcal{S}$ is denoted by $\mathcal{C}$. It is well-known that a function $f \in \mathcal{A}$ is in $\mathcal{C}$ if, and only if, ${\rm Re}\left(1 + ({zf''(z)}{f'(z)})\right) > 0$ for $z \in \mathbb{D}$. For more details about these classes, we refer to \cite{Duren-1983,Goodman-1983}.
\begin{defi}
	For two functions $f$ and $g$ in $\mathcal{H}$, we say that $f$ is subordinate to $g$, written as $f \prec g$, if there exists a function $\omega \in \mathcal{H}$ with $\omega(0) = 0$ and $|\omega(z)| < 1$ such that $f(z) = g(\omega(z))$ for $z \in \mathbb{D}$.
\end{defi}
 Let $\varphi$ be an analytic univalent function with positive real part in $\mathbb{D}$ such that $\varphi(\mathbb{D})$ is symmetric with respect to the real axis and starlike with respect to $\varphi(0) = 1$ and $\varphi'(0) > 0$. For such a function $\varphi$, Ma and Minda \cite{Ma-Minda-1992} introduced the classes $\mathcal{S}^*(\varphi)$ and $\mathcal{C}(\varphi)$ as
\begin{align*}
	\mathcal{S}^*(\varphi) = \left\{ f \in \mathcal{A} : \frac{zf'(z)}{f(z)} \prec \varphi(z) \right\} \; \text{and} \; \mathcal{C}(\varphi) = \left\{ f \in \mathcal{A} : 1 + \frac{zf''(z)}{f'(z)} \prec \varphi(z) \right\},
\end{align*}
respectively. Sometimes $\mathcal{S}^*(\varphi)$ and $\mathcal{C}(\varphi)$ are called Ma-Minda classes of starlike and convex functions, respectively. One can easily verify the inclusion relations $\mathcal{S}^*(\varphi) \subset \mathcal{S}^*$ and $\mathcal{C}(\varphi) \subset \mathcal{C}$. It is important to note that $f \in \mathcal{S}^*(\varphi)$ if, and only if, $\mathcal{J}[f] \in \mathcal{C}(\varphi)$, where $\mathcal{J}[f]$ is the Alexander transformation of $f$ defined by
\begin{equation*}
	\mathcal{J}[f](z) = \int_0^z \frac{f(t)}{t} dt = f(z) * (-\log(1-z)).
\end{equation*}
Here, the symbol $*$ denotes the \textit{Hadamard product} (or convolution) of two power series.\vspace{1.2mm}

Different choices of the function $\varphi$ gives several well-known geometric subclasses of $\mathcal{A}$. For instance, if we take $\varphi(z) = (1+z) / (1-z)$, the classes $\mathcal{S}^*(\varphi)$ and $\mathcal{C}(\varphi)$ reduce to the standard classes $\mathcal{S}^*$ of starlike functions and $\mathcal{C}$ of convex functions, respectively. For $\varphi(z) = (1+(1-2\alpha)z) / (1-z)$ with $0 \le \alpha < 1$, we obtain the classes $\mathcal{S}^*(\alpha)$ and $\mathcal{C}(\alpha)$ of starlike and convex functions of order $\alpha$. When $\varphi(z) = \left((1+z)/(1-z)\right)^\gamma$ for $0 < \gamma \le 1$, the classes $\mathcal{S}^*(\varphi)$ and $\mathcal{C}(\varphi)$ are denoted by $\mathcal{S}^*(\gamma)$ and $\mathcal{C}(\gamma)$, which are known as the strongly starlike and strongly convex functions of order $\gamma$. Furthermore, for $\varphi(z) = (1+Az) / (1+Bz)$ with $-1 \le B < A \le 1$, we obtain the classes of Janowski starlike and Janowski convex functions, denoted by $\mathcal{S}^*(A, B)$ and $\mathcal{C}(A, B)$, respectively. For more details, we refer to \cite{Janowski-1973,Raina-Sokol-2015,Ronning-1993}.\vspace{2mm}

In the present article, we consider the class of functions $\mathcal C_{e^\lambda} := \mathcal{C}(\varphi)$ with $\varphi(z) = e^{\lambda z}$ for $0 < \lambda \le \pi/2$. More precisely, this class is defined as
\begin{align*}
	\mathcal C_{e^\lambda} = \left\{ f \in \mathcal{A} : 1 + \frac{zf''(z)}{f'(z)} \prec e^{\lambda z} \right\}.
\end{align*}
For $\lambda = 1$, the class $\mathcal C_{e^\lambda} = \mathcal{C}_e = \mathcal{C}(e^z)$ was studied by Mendiratta \textit{et al.} in \cite{Mendiratta-etal-2014}. The parameterized class \(\mathcal C_{e^\lambda}\) was subsequently investigated by by Shi \textit{et al.} in \cite{Shi-etal-2020}.\vspace{1.2mm}

It is easy to see that a function $f$ in $\mathcal C_{e^\lambda}$ satisfies the condition
\begin{align*}
	\left| \log\left(1 + \frac{zf''(z)}{f'(z)}\right) \right| \le \lambda, \quad z \in \mathbb{D}.
\end{align*}
Furthermore, it is clear that a function $f \in \mathcal C_{e^\lambda}$ if, and only if, there exists a function $p_1 \in \mathcal{H}$ with $p_1(z) \prec e^{\lambda z}$ ($0 < \lambda \le \pi/2$) such that
\begin{align*}
		1 + \frac{zf''(z)}{f'(z)} = p_1(z).
\end{align*}
This representation will be used throughout the proof.\vspace{2mm}

Although the class \(\mathcal C_{e^\lambda}\) and related coefficient problems have been considered previously, the present work addresses a different set of extremal questions. In particular, we investigate the growth and distortion of \(f\) and \(f'\), together with the pre-Schwarzian and Schwarzian norms, and derive parameter-dependent estimates in terms of the exponential parameter \(\lambda\) and, where appropriate, the initial coefficient \(f''(0)\). The extremal problems are formulated directly through the associated Schwarz functions, allowing the dependence of the estimates on \(\lambda\) to be made explicit.\vspace{1.2mm} 

Moreover, we consider the class $\mathcal C_{e^\lambda}$ in our study, because it provides a natural and geometrically rich framework for studying convex functions bounded by exponential mappings. Geometrically, while classical convex functions constrain the curvature operator $1 + zf''(z)/f'(z)$ to the right half-plane $\mathbb{H} = \{ w \in \mathbb{C} : \text{Re}(w) > 0 \}$, functions in $\mathcal C_{e^\lambda}$ restrict this operator to an exponentially constrained domain $\Omega_\lambda = \{ e^{\lambda z} : z \in \mathbb{D} \}$. For $0<\lambda\le\pi/2$, the image domain $\Omega_\lambda=e^{\lambda\mathbb{D}}$ is a smooth domain symmetric with respect to the real axis and contained in the right half-plane, see, Figure \ref{Fig-1}. This exponential containment imposes a subtle asymmetry on the boundary rotation of $f(\mathbb{D})$, leading to refined growth, distortion, and Schwarzian norm estimates that significantly sharpen classical bounds for convex functions.\vspace{1.2mm}

\begin{figure}[htbp]
	\centering
	\includegraphics[width=0.95\textwidth]{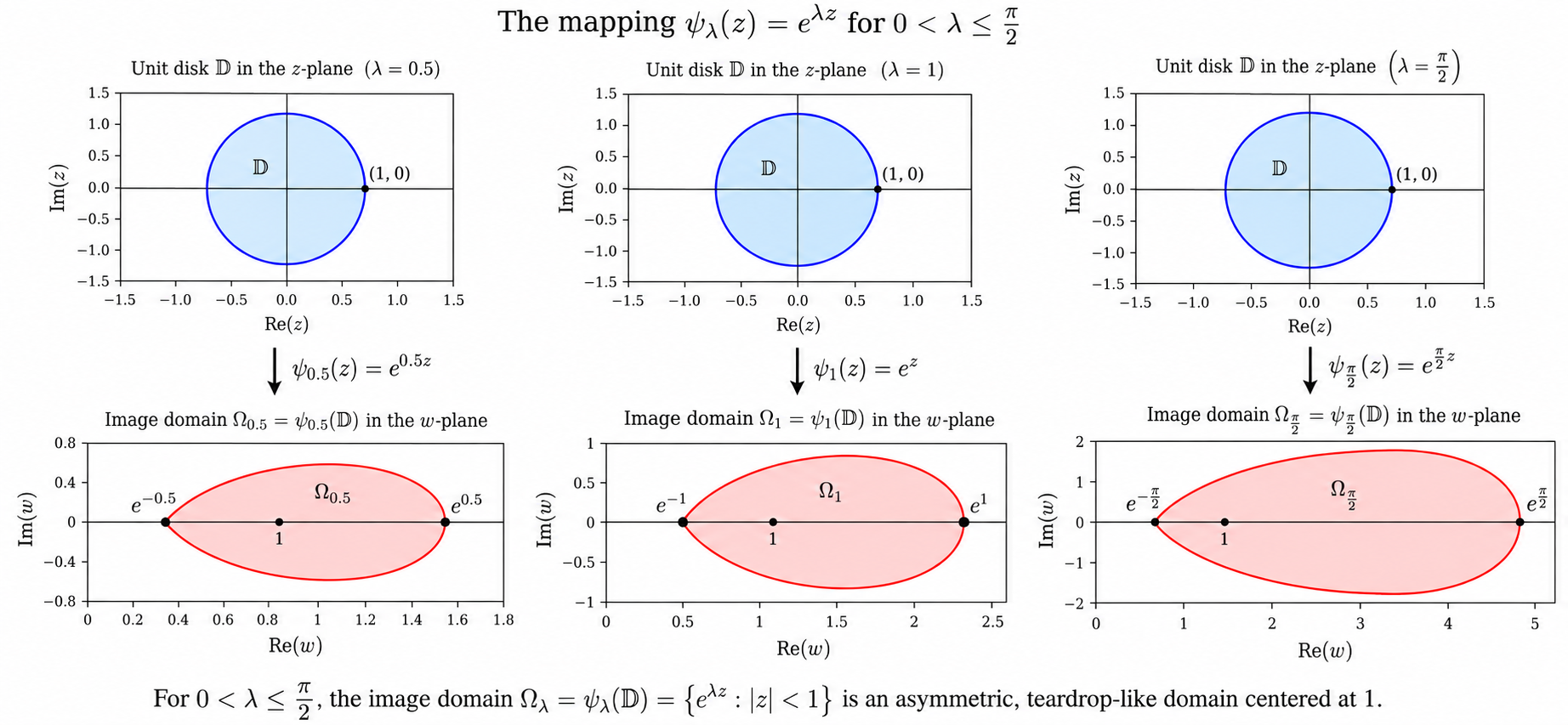}
	\caption{The image domain $\Omega_\lambda=\psi_\lambda(\mathbb{D})={e^{\lambda z}:z\in\mathbb{D}}$ under the mapping $\psi_\lambda(z)=e^{\lambda z}$ for three values of $\lambda$. For $\lambda=0.5$ and $\lambda=1$, the domains $\Omega_{0.5}$ and $\Omega_1$ exhibit progressively stronger deformation from the unit-disk geometry, while remaining symmetric with respect to the real axis. At the critical value $\lambda=\pi/2$, the domain $\Omega_{\pi/2}$ has boundary points $i$ and $-i$, corresponding to $z=i$ and $z=-i$, respectively, and $\Omega_{\pi/2}\cap\mathbb{R}=(e^{-\pi/2},e^{\pi/2})$. In particular, for $0<\lambda\leq\pi/2$, the principal argument satisfies $|\arg w|<\lambda$ for $w\in\Omega_\lambda$.}
	\label{Fig-1}
\end{figure}
The Schur class is defined as
\begin{equation*}
	\mathcal{SS} = \{f : \mathbb{D} \to \mathbb{D} : f \text{ is analytic}\}.
\end{equation*}
By the maximum modulus principle we know that $f\in \mathbb{T} := \partial\mathbb{D}$ for some $z \in \mathbb{D}$ if, and only if, $f$ is a constant function with $|f(z)| = 1$. A fundamental cornerstone of complex analysis is the Schwarz lemma. 
\begin{lemA}\emph{(Schwarz lemma)}
	If $f \in \mathcal{SS}$ with $f(0) = 0$, then
	\begin{enumerate}
		\item[(i)] $|f(z)| \le |z|$ for all $z \in \mathbb{D}$, and
		\item[(ii)] $|f'(0)| \le 1$.
	\end{enumerate}
	Moreover, $|f(w)| = |w|$ for some $w \in \mathbb{D} \setminus \{0\}$, and if $|f'(0)| = 1$, then there is $\eta \in \mathbb{T}$ such that $f(z) = \eta z$ for all $z \in \mathbb{D}$.
\end{lemA}
The main results of this paper are organized around three related aspects of the exponentially convex class $\mathcal C_{e^\lambda}$. First, we establish an analytic characterization of $\mathcal C_{e^\lambda}$ through the associated Schwarz function, which provides the basic representation used throughout the paper. Second, we derive sharp parameter-dependent growth and distortion estimates for $f$ and $f'$, together with a sharp estimate for the pre-Schwarzian norm. Third, by obtaining an explicit representation of the Schwarzian derivative in terms of the underlying Schwarz function, we establish sharp Schwarzian norm estimates both for the subclass $\mathcal C_{e^\lambda,0}$ satisfying $f''(0)=0$ and for the general class $\mathcal C_{e^\lambda}$. The proofs combine differential subordination, Schwarz--Pick estimates, and suitable extremal Schwarz functions, thereby making explicit the dependence of the resulting bounds on the parameter $\lambda$ and, in the general Schwarzian estimate, on the initial coefficient $f''(0)$.
\section{\bf Main Results}
Let $\mathbb{B}$ denote the set of all functions $\phi$ that are analytic in $\mathbb{D}$ and satisfy $|\phi(z)| \le 1$ for all $z \in \mathbb{D}$. We also consider the subfamily $\mathbb{B}_0 := \{\phi \in \mathbb{B} : \phi(0) = 0\}$.\vspace{2mm}

In this section, motivated by recent developments by Ahamed \emph{et al.}  \cite{Ahamed-Allu-Hossain-MM-2025}, we utilize the Schwarz lemma to examine the subordination structure of the exponentially convex class $\mathcal C_{e^\lambda}$. First, an equivalent characterization of $\mathcal C_{e^\lambda}$ is given in Theorem \ref{Th-2.1}. Next, Theorem \ref{Th-2.2} provides distortion and growth estimates, together with sharp bounds for the pre-Schwarzian and Schwarzian norms in terms of $f''(0)$. As applications, these results yield natural extensions for the classical class $\mathcal{C}$ of convex functions. \vspace{1.2mm}

We begin by establishing an equivalent analytic characterization for functions belonging to the exponentially convex class $\mathcal C_{e^\lambda}$. This structural representation forms the foundation for deriving the subsequent norm bounds and growth estimates.

\begin{thm}\label{Th-2.1}
	Let $0 < \lambda \le \pi/2$ and $f \in \mathcal{A}$. Then the following statements are equivalent
	\begin{enumerate}
		\item[(i)] $f \in \mathcal C_{e^\lambda}$;\vspace{1.2mm}
		\item[(ii)]  $\displaystyle\left| \log\left(1 + \frac{zf''(z)}{f'(z)}\right) \right| \le \lambda |z|$ for all $z \in \mathbb{D}$;\vspace{1.2mm}
		\item[(iii)]  $\displaystyle(1 - |z|^2) \left| \frac{f''(z)}{f'(z)} \right| \le \frac{1 - |z|^2}{|z|} \left( e^{\lambda |z|} - 1 \right)$ for all $z \in \mathbb{D} \setminus \{0\}$.
	\end{enumerate}
	Furthermore, the bounds in {\rm (ii)} and {\rm (iii)} are sharp for each $z \in \mathbb{D} \setminus \{0\}$, with equality attained by the function $f_\lambda \in \mathcal C_{e^\lambda}$ defined by
	\begin{align}\label{Eq-2.1}
		f_\lambda(z) &= \int_{0}^{z} \exp\left( \int_{0}^{t} \frac{e^{\lambda \xi} - 1}{\xi} \, d\xi \right) dt\\&=z + \frac{\lambda}{2} z^2 + \frac{\lambda^2}{4} z^3 + \frac{17\lambda^3}{144} z^4 + \frac{71\lambda^4}{960} z^5 + \mathcal{O}(z^6).\nonumber
	\end{align}
\end{thm}
\begin{proof}[\bf Proof of Theorem \ref{Th-2.1}]
	$\text{(i)} \implies \text{(ii)}$: Let $f \in \mathcal C_{e^\lambda}$. By the definition of differential subordination, there exists a Schwarz function $\omega \in \mathbb{B}_0$ (\textit{i.e.,} analytic in $\mathbb{D}$ with $\omega(0) = 0$ and $|\omega(z)| < 1$ for $z \in \mathbb{D}$) such that
	\begin{align*}
		1 + \frac{z f''(z)}{f'(z)} = e^{\lambda \omega(z)}, \quad z \in \mathbb{D}.
	\end{align*}
	By Schwarz's lemma, $\omega(z) = z\phi(z)$ for some $\phi \in \mathbb{B}$ satisfying $|\phi(z)| \le 1$ for all $z \in \mathbb{D}$. Consequently, we have
	\begin{align*}
		\log\left(1 + \frac{z f''(z)}{f'(z)}\right) = \lambda z \phi(z).
	\end{align*}
	Taking the absolute value on both sides and utilizing the estimate $|\phi(z)| \le 1$ yields
\begin{align*}
	\left| \log\left(1 + \frac{z f''(z)}{f'(z)}\right) \right| = \lambda |z| |\phi(z)| \le \lambda |z|,
\end{align*}
	which proves (ii).\vspace{1.2mm}
	
	$\text{(ii)} \implies \text{(i)}$: Define $\phi: \mathbb{D} \to \mathbb{C}$ by
	\begin{align*}
		\phi(z) = 
		\begin{cases} 
			\displaystyle\frac{1}{\lambda z} \log\left(1 + \frac{z f''(z)}{f'(z)}\right), & z \in \mathbb{D} \setminus \{0\}, \\[1.5ex]
			\displaystyle\frac{f''(0)}{\lambda}, & z = 0.
		\end{cases}
	\end{align*}
	Since $f \in \mathcal{A}$, $\phi$ is analytic in $\mathbb{D}$. Condition (ii) implies $|\phi(z)| \le 1$ for all $z \in \mathbb{D}$. Defining $\omega(z) = z\phi(z)$, we see that $\omega \in \mathbb{B}_0$, and hence $1 + zf''(z)/f'(z) = e^{\lambda \omega(z)} \prec e^{\lambda z}$, establishing (i).
	
	$\text{(ii)} \implies \text{(iii)}$: From $1 + zf''(z)/f'(z) = e^{\lambda z \phi(z)}$ with $|\phi(z)| \le 1$, we obtain
	\begin{align*}
		\frac{f''(z)}{f'(z)} = \frac{e^{\lambda z \phi(z)} - 1}{z}, \quad z \in \mathbb{D} \setminus \{0\}.
	\end{align*}
	Applying the inequality $|e^w - 1| \le e^{|w|} - 1$, valid for all $w \in \mathbb{C}$, alongside $|\phi(z)| \le 1$, we deduce
	\begin{align*}
		\left| \frac{f''(z)}{f'(z)} \right| \le \frac{e^{\lambda |z| |\phi(z)|} - 1}{|z|} \le \frac{e^{\lambda |z|} - 1}{|z|}.
	\end{align*}
	Multiplying both sides by $(1 - |z|^2)$ gives statement (iii).\vspace{1.2mm}
	
	$\text{(iii)} \implies \text{(ii)}$: Conversely, assume (iii) holds. Dividing both sides of (iii) by $(1 - |z|^2)$ and passing to the limit as $|z| \to 0^+$, we note that $f''(0)/\lambda = \lim_{z \to 0} f''(z)/(\lambda f'(z))$ satisfies $|f''(0)| \le \lambda$. For any $z \in \mathbb{D} \setminus \{0\}$, write $g(z) = \log(1 + zf''(z)/f'(z))$. Since $|e^w - 1| \ge |w| - \frac{1}{2}|w|^2 \dots$, the bound in (iii) restricts the growth of $g(z)$ on disks $\mathbb{D}_r = \{z : |z| < r\}$, which by the maximum modulus principle guarantees $|g(z)| \le \lambda |z|$ for all $z \in \mathbb{D}$.\vspace{1.2mm}
	
	Finally, to verify the sharpness of (ii) and (iii), consider the extremal function $f_0 \in \mathcal C_{e^\lambda}$,  given by \eqref{Eq-2.1}, corresponding to $\phi(z) \equiv 1$. For $f_0$, a direct calculation shows
	\begin{align*}
		1 + \frac{z f_0''(z)}{f_0'(z)} = e^{\lambda z}.
	\end{align*}
	For any point $z = r \in (0,1)$, we have the estimates
	\begin{align*}
		\left| \log\left(1 + \frac{r f_0''(r)}{f_0'(r)}\right) \right| = \log(e^{\lambda r}) = \lambda r,
	\end{align*}
	and
	\begin{align*}
		(1 - r^2) \left| \frac{f_0''(r)}{f_0'(r)} \right| = \frac{1 - r^2}{r} (e^{\lambda r} - 1),
	\end{align*}
	which proves that equality is attained in both (ii) and (iii) along the positive real axis.
\end{proof}
Sharp growth and distortion theorems for subclasses of analytic functions have attracted considerable attention in geometric function theory (see, e.g., \cite{Ahamed-Allu-Hossain-MM-2025,Kim-etal-2004}). In this paper, we establish sharp growth and distortion bounds for functions in the class $\mathcal C_{e^\lambda}$. The extremal function is obtained by integrating the corresponding differential subordination relation.
\begin{thm}\label{Th-2.2}
	If $f \in \mathcal C_{e^\lambda}$ and $0 < \lambda \le \pi/2$, then for all $z \in \mathbb{D}$,
	$$\exp\left(-\int_0^{|z|} \frac{e^{\lambda \xi} - 1}{\xi} d\xi\right) \le |f'(z)| \le \exp\left(\int_0^{|z|} \frac{e^{\lambda \xi} - 1}{\xi} d\xi\right)$$
	and
	$$\int_0^{|z|} \exp\left(-\int_0^{\eta} \frac{e^{\lambda \xi} - 1}{\xi} d\xi\right) d\eta \le |f(z)| \le \int_0^{|z|} \exp\left(\int_0^{\eta} \frac{e^{\lambda \xi} - 1}{\xi} d\xi\right) d\eta.$$
	All the bounds are sharp. Equality holds for the function $f_\lambda$ given by \eqref{Eq-2.1}.
\end{thm}

\begin{proof}[\bf Proof of Theorem \ref{Th-2.2}]
	Let $f \in \mathcal C_{e^\lambda}$. By definition of differential subordination, there exists a Schwarz function $\phi$ analytic in $\mathbb{D}$ with $|\phi(z)| \le 1$ such that 
	\begin{equation}\label{Eq-2.2}
		1 + \frac{zf''(z)}{f'(z)} = e^{\lambda z \phi(z)}.
	\end{equation}
	Writing $z = r e^{i\theta}$ ($0 < r < 1$), we have $| \lambda z \phi(z) | \le \lambda r$. Since $|e^w - 1| \le e^{|w|} - 1$ for all $w \in \mathbb{C}$, it follows from \eqref{Eq-2.2} that
	\begin{equation*}
		\left| \frac{zf''(z)}{f'(z)} \right| = \left| e^{\lambda z \phi(z)} - 1 \right| \le e^{\lambda r} - 1.
	\end{equation*}
	Consequently, in view of the inequality $-|w| \le \operatorname{Re}(w) \le |w|$, it is easy to see that
	\begin{equation*}
		-(e^{\lambda r} - 1) \le \operatorname{Re}\left(\frac{zf''(z)}{f'(z)}\right) \le e^{\lambda r} - 1,
	\end{equation*}
	which gives the estimate
	\begin{equation}\label{Eqq-2.3}
		-\frac{e^{\lambda r} - 1}{r} \le \frac{\partial}{\partial r} \log |f'(re^{i\theta})| \le \frac{e^{\lambda r} - 1}{r}.
	\end{equation}
	Integrating both sides of \eqref{Eqq-2.3} with respect to $r$ from $0$ to $|z|$, we obtain
	\begin{equation}\label{Eqq-2.4}
		-\int_0^{|z|} \frac{e^{\lambda \xi} - 1}{\xi} d\xi \le \log |f'(z)| \le \int_0^{|z|} \frac{e^{\lambda \xi} - 1}{\xi} d\xi.
	\end{equation}
	Exponentiating the inequality \eqref{Eqq-2.4}, we obtain the required sharp two-sided distortion bounds
	\begin{equation*}
		\exp\left(-\int_0^{|z|} \frac{e^{\lambda \xi} - 1}{\xi} d\xi\right) \le |f'(z)| \le \exp\left(\int_0^{|z|} \frac{e^{\lambda \xi} - 1}{\xi} d\xi\right).
	\end{equation*}
	Next, for the growth bounds, integrating the upper estimate along the radial segment from $0$ to $z = |z|e^{i\theta}$ gives
	\begin{equation*}
		|f(z)| = \left| \int_0^{|z|} f'(\eta e^{i\theta}) e^{i\theta} d\eta \right| \le \int_0^{|z|} |f'(\eta e^{i\theta})| d\eta \le \int_0^{|z|} \exp\left(\int_0^{\eta} \frac{e^{\lambda \xi} - 1}{\xi} d\xi\right) d\eta.
	\end{equation*}
	For the lower bound of $|f(z)|$, let $d = |f(z)|$. The image $f(\mathbb{D})$ contains a line segment $\Gamma$ joining $0$ to $f(z)$ of length $d$. The preimage $\gamma = f^{-1}(\Gamma)$ is an arc in $\mathbb{D}$ connecting $0$ to $z$. Thus,
	\begin{equation*}
		|f(z)| = \int_{\Gamma} |dw| = \int_{\gamma} |f'(\zeta)| |d\zeta| \ge \int_0^{|z|} \exp\left(-\int_0^{\eta} \frac{e^{\lambda \xi} - 1}{\xi} d\xi\right) d\eta.
	\end{equation*}
	
	The sharpness of all bounds follows directly from the function $f_\lambda \in \mathcal C_{e^\lambda}$ defined by \eqref{Eq-2.1}, corresponding to the extremal Schwarz function $\phi(z) \equiv 1$. Equality in the upper bounds is attained at $z = r > 0$, while equality in the lower bounds is attained at $z = -r < 0$.
\end{proof}
Next, we find the sharp bounds of the pre-Schwarzian and Schwarzian norms for the class $\mathcal C_{e^\lambda}$ with respect to $f''(0)$, assuming that $f''(0) = 0$. The following lemma, adapted from, is a key tool in proving these results.\vspace{2mm}

\begin{lemB}\emph{\cite{Carrasco-Hernandez-AMP-2023}}
	If $\phi(z): \mathbb{D} \to \mathbb{D}$ is an analytic function, then  
	\begin{equation}
		\frac{|\phi(z)|^2}{1 - |\phi(z)|^2} \le \frac{(|\phi(0)| + |z|)^2}{(1 - |\phi(0)|^2)(1 - |z|^2)}.
	\end{equation}
\end{lemB}
Let $\Omega$ be a simply connected domain in the complex plane $\mathbb{C}$ containing at least two boundary points. For a locally univalent analytic function $f$ defined on $\Omega$, the \emph{pre-Schwarzian derivative} $P_f$ and the \emph{Schwarzian derivative} $S_f$ are defined by (see \cite{Kim-etal-2004,Okuyama-2000})
\begin{align*}
	P_f(z) = \frac{f''(z)}{f'(z)}
\end{align*}
and
\begin{align*}
	S_f(z) = (P_f)'(z) - \frac{1}{2}\left(P_f(z)\right)^2 = \frac{f'''(z)}{f'(z)} - \frac{3}{2}\left(\frac{f''(z)}{f'(z)}\right)^2,
\end{align*}
respectively. Let $\lambda_{\Omega}(z)|dz|$ denote the hyperbolic (Poincar\'e) metric on $\Omega$ with constant Gaussian curvature $-4$. The \emph{pre-Schwarzian norm} $\|P_f\|_{\Omega}$ and the \emph{Schwarzian norm} $\|S_f\|_{\Omega}$ of $f$ on $\Omega$ are defined by
\begin{align*}
	\|P_f\|_{\Omega} = \sup_{z \in \Omega} |P_f(z)| \lambda_{\Omega}^{-1}(z) \quad \text{and} \quad \|S_f\|_{\Omega} = \sup_{z \in \Omega} |S_f(z)| \lambda_{\Omega}^{-2}(z).
\end{align*}
In particular, if $\Omega$ is the open unit disk $\mathbb{D} = \{z \in \mathbb{C} : |z| < 1\}$, then the Poincar\'e density is given by $\lambda_{\mathbb{D}}(z) = (1 - |z|^2)^{-1}$. In this case, the norms are simply written as $\|P_f\|$ and $\|S_f\|$, which take the forms
\begin{align*}
	\|P_f\| = \sup_{z \in \mathbb{D}} (1 - |z|^2) \left| \frac{f''(z)}{f'(z)} \right| \quad \text{and} \quad \|S_f\| = \sup_{z \in \mathbb{D}} (1 - |z|^2)^2 |S_f(z)|.
\end{align*}
These derivatives and their associated hyperbolic norms possess several fundamental transformation and invariance properties:\vspace{1.2mm}

\noindent{\bf [I].} For locally univalent analytic functions $f$ and $g$ for which the composition $f \circ g$ is well defined, the differential operators satisfy the classical composition formulas:
	\begin{align*}
		P_{f \circ g}(z) &= P_f(g(z)) g'(z) + P_g(z), \\
		S_{f \circ g}(z) &= S_f(g(z)) (g'(z))^2 + S_g(z).
	\end{align*}
	
	\noindent{\bf [II].} If $T(z) = \frac{az + b}{cz + d}$ ($ad - bc \neq 0$) is a non-constant M\"obius transformation, then $S_T(z) \equiv 0$. Consequently, the Schwarzian derivative exhibits M\"obius invariance: $S_{T \circ f}(z) = S_f(z)$. In contrast, for $T(z) = az + b$ ($a \neq 0$), $P_T(z) \equiv 0$, rendering the pre-Schwarzian derivative invariant under affine transformations: $P_{T \circ f}(z) = P_f(z)$.\vspace{2mm}
	
	\noindent{\bf [III].} Let $\text{Aut}(\mathbb{D})$ denote the group of conformal automorphisms of $\mathbb{D}$. The pre-Schwarzian and Schwarzian norms on $\mathbb{D}$ are invariant under domain automorphisms; that is, for any $\phi \in \text{Aut}(\mathbb{D})$,
	\begin{align*}
		\|P_{f \circ \phi}\| = \|P_f\| \quad \text{and} \quad \|S_{f \circ \phi}\| = \|S_f\|.
	\end{align*}
	
	\noindent{\bf [IV].} The norms $\|P_f\|$ and $\|S_f\|$ play a crucial role in univalence criteria
	\begin{itemize}
		\item[(a).] \emph{Becker's Criterion:} If $\|P_f\| \le 1$, then $f$ is univalent in $\mathbb{D}$.
		\item[(b).] \emph{Nehari's Criterion:} If $\|S_f\| \le 2$, then $f$ is univalent in $\mathbb{D}$. Conversely, if $f$ is univalent in $\mathbb{D}$, then $\|S_f\| \le 6$.
		\item[(c).] A function $f$ is uniformly locally univalent in $\mathbb{D}$ if, and only if, $\|P_f\| < \infty$ (or equivalently, $\|S_f\| < \infty$).
	\end{itemize}
It is well-known that the pre-Schwarzian norm $||P_f||\leq 6$ holds for the univalent analytic function $f$ is defined in $\mathbb{D}$. In $1972$, Becker \cite{Becker-JRAM-1983} used the pre-Schwarzian derivative to obtain the sufficient condition that the function in $\mathbb{D}$ is univalent, in other words, if $||P_f||\leq1$, then the function $f$ is univalent in $\mathbb{D}$. In $1976$, Yamashita \cite{Yamashita-1976} proved that $||P_f||$ is finite if, and only if, $f$ is uniformly locally univalent in $\mathbb{D}$, \emph{i.e.}, there exists a constant $\rho$ such that $f$ is univalent on the hyperbolic disk $|(z-a)/(1-\bar{a}z)|<\tanh\rho$ of radius $\rho$ for every $a\in\mathbb{D}$. Sugawa \cite{Sugawa-1998} studied the strongly starlike functions of order $\alpha\; (0<\alpha\leq1)$. Yamashita\cite{Yamashita-HMJ-1999} generalized sugawa's results by a general class named Gelfer-starlike of exponential order $\alpha (\alpha>0)$ and the Gelfer-close-to-convex of exponential order $(\alpha,\beta)$ ($\alpha>0$, $\beta>0$). These Gelfer classes also contain the classical starlike, convex, close-to-convex all of order $\alpha$ ($0\leq\alpha<1$), which are denote by $\mathcal{S^*(\alpha)}$, $\mathcal{C(\alpha)}$, $\mathcal{K(\alpha)}$ respectively (see \cite{Ponnusamy-Sahoo-2008,Sokol-Stankiewicz-1996,Kumar-Ravichandran-2018,Kanas-2006,Kim-Sugawa-2002})\vspace{2mm}

The following theorem provides the sharp bound for the pre-Schwarzian norm of functions in the class $\mathcal C_{e^\lambda}$.
\begin{thm}\label{Th-2.3}
	Let $0 < \lambda \le \pi/2$ and $f \in \mathcal C_{e^\lambda}$. Then the pre-Schwarzian norm of $f$ satisfies
	\begin{align}\label{Eq-2.3}
		\|P_f\| &= \sup_{z \in \mathbb{D}} (1 - |z|^2) \left| \frac{f''(z)}{f'(z)} \right| \\&\nonumber\le 
		\begin{cases} 
			\lambda, & \text{if } g'(r) \le 0 \text{ for all } r \in (0,1), \\[1.5ex]
			g(r_0), & \text{where } r_0 \in (0,1) \text{ is the unique root of } g'(r) = 0,
		\end{cases}
	\end{align}
	where $g:(0,1) \to \mathbb{R}$ is defined by
	\begin{align*}
		g(r) = \frac{1 - r^2}{r} \left(e^{\lambda r} - 1\right).
	\end{align*}
	The estimate \eqref{Eq-2.3} is sharp.
\end{thm}
\begin{proof}[\bf Proof of Theorem \ref{Th-2.3}]
	Let $f \in \mathcal C_{e^\lambda}$. By the definition of differential subordination, there exists an analytic Schwarz function $\phi \in \mathbb{B}$ with $|\phi(z)| \le 1$ for all $z \in \mathbb{D}$ such that
	\begin{align*}
		1 + \frac{z f''(z)}{f'(z)} = e^{\lambda z \phi(z)}, \quad z \in \mathbb{D}.
	\end{align*}
	Rearranging this differential relation yields the explicit form of the pre-Schwarzian derivative:
	\begin{align*}
		\frac{f''(z)}{f'(z)} = \frac{e^{\lambda z \phi(z)} - 1}{z}, \quad z \in \mathbb{D} \setminus \{0\}.
	\end{align*}
	Multiplying both sides by $(1 - |z|^2)$, taking the absolute value, and using the inequality $|e^w - 1| \le e^{|w|} - 1$ along with $|\phi(z)| \le 1$, we obtain for $z \in \mathbb{D} \setminus \{0\}$:
	\begin{align*}
		(1 - |z|^2) \left| \frac{f''(z)}{f'(z)} \right| &= (1 - |z|^2) \left| \frac{e^{\lambda z \phi(z)} - 1}{z} \right| \\
		&\le \frac{1 - |z|^2}{|z|} \left( e^{\lambda |z| |\phi(z)|} - 1 \right) \\
		&\le \frac{1 - |z|^2}{|z|} \left( e^{\lambda |z|} - 1 \right) := g(|z|).
	\end{align*}
	Taking the supremum over all $z \in \mathbb{D}$, it follows that
	\begin{align*}
		\|P_f\| \le \sup_{0 \le r < 1} g(r).
	\end{align*}
	To evaluate $\displaystyle\sup_{0 \le r < 1} g(r)$, we analyze the behavior of $g(r)$ on $(0,1)$. Using the Taylor expansion 
	\begin{align*}
		e^{\lambda r} - 1 = \lambda r + \frac{\lambda^2 r^2}{2} + \mathcal{O}(r^3)
	\end{align*} as $r \to 0^+$, we find
	\begin{align*}
		\lim_{r \to 0^+} g(r) = \lim_{r \to 0^+} \frac{1 - r^2}{r} \left( \lambda r + \frac{\lambda^2 r^2}{2} + \mathcal{O}(r^3) \right) = \lambda.
	\end{align*}
	Moreover, $\lim_{r \to 1^-} g(r) = 0$. Differentiating $g(r)$ with respect to $r$ gives
	\begin{align*}
		g'(r) = \frac{1}{r^2} \left[ (1 - r^2)\lambda r e^{\lambda r} - (1 + r^2)\left(e^{\lambda r} - 1\right) \right].
	\end{align*}
	If $g'(r) \le 0$ for all $r \in (0,1)$, then $g$ is strictly decreasing on $(0,1)$, and its supremum is attained at the origin, giving 
	\begin{align*}
		\sup_{0 \le r < 1} g(r) = \lambda.
	\end{align*} Otherwise, continuous differentiation implies that $g(r)$ attains a local maximum at a unique critical point $r_0 \in (0,1)$ satisfying $g'(r_0) = 0$, which gives
	\begin{align*}
		\sup_{0 \le r < 1} g(r) = g(r_0) = \frac{1 - r_0^2}{r_0} (e^{\lambda r_0} - 1).
	\end{align*}
	To verify the sharpness of \eqref{Eq-2.3}, consider the function $f_\lambda \in \mathcal C_{e^\lambda}$ defined by $1 + z f_\lambda''(z)/f_\lambda'(z) = e^{\lambda z}$, corresponding to $\phi(z) \equiv 1$. For this function, choosing $z = r \in (0,1)$ gives
	\begin{align*}
		\left(1 - r^2\right) \left| \frac{f_\lambda''(r)}{f_\lambda'(r)} \right| = \frac{1 - r^2}{r} \left( e^{\lambda r} - 1 \right) := g(r).
	\end{align*}
	Taking the supremum over $r \in (0,1)$ shows that $\|P_{f_\lambda}\|$ precisely matches the upper bound in \eqref{Eq-2.3}, completing the proof.
\end{proof}

We first derive an explicit representation for $S_f(z)$ in terms of the underlying Schwarz function.

\begin{prop}\label{Prop-Sf-formula}
	Let $0 < \lambda \le \pi/2$ and $f \in \mathcal C_{e^\lambda}$. Then there exists an analytic Schwarz function $\omega \in \mathbb{B}$ such that
	\begin{equation}\label{Eq-Sf-explicit}
		S_f(z) = \frac{\lambda (\omega(z) + z\omega'(z)) e^{\lambda z \omega(z)} z - (e^{\lambda z \omega(z)} - 1)}{z^2} - \frac{1}{2}\left(\frac{e^{\lambda z \omega(z)} - 1}{z}\right)^2, \quad z \in \mathbb{D} \setminus \{0\}.
	\end{equation}
	Furthermore, $S_f(z)$ satisfies the point-wise inequality
	\begin{equation}\label{Eq-Sf-pointwise}
		(1 - |z|^2)^2 |S_f(z)| \le 2\lambda |\omega'(z) + (1 - \lambda)\omega^2(z)| \frac{(1 - |z|^2)^2}{|1 - z\omega(z)|^2}.
	\end{equation}
\end{prop}
\begin{proof}[\bf Proof of Proposition \ref{Prop-Sf-formula}]
	Let $f \in \mathcal C_{e^\lambda}$. By definition of differential subordination, there exists a Schwarz function $\omega \in \mathbb{B}$ (i.e., analytic in $\mathbb{D}$ with $|\omega(z)| \le 1$) such that
	\begin{equation*}
		1 + \frac{z f''(z)}{f'(z)} = e^{\lambda z \omega(z)}, \quad z \in \mathbb{D}.
	\end{equation*}
	Rearranging this equation gives the pre-Schwarzian derivative
	\begin{equation}\label{Eq-pre-Schwarzian-omega}
		P_f(z) = \frac{f''(z)}{f'(z)} = \frac{e^{\lambda z \omega(z)} - 1}{z}, \quad z \in \mathbb{D} \setminus \{0\}.
	\end{equation}
	Differentiating \eqref{Eq-pre-Schwarzian-omega} with respect to $z$, we obtain
	\begin{equation*}
		P_f'(z) = \left(\frac{f''(z)}{f'(z)}\right)' = \frac{\lambda (\omega(z) + z\omega'(z)) e^{\lambda z \omega(z)} z - (e^{\lambda z \omega(z)} - 1)}{z^2}.
	\end{equation*}
	By definition, the Schwarzian derivative is given by 
	\begin{align*}
		S_f(z) = P_f'(z) - \frac{1}{2}(P_f(z))^2
	\end{align*} which immediately yields the explicit formula \eqref{Eq-Sf-explicit}.\vspace{1.2mm}
	
	Applying the triangle inequality along with the Schwarz--Pick inequality 
	\begin{align*}
		|\omega'(z)| \le \frac{1 - |\omega(z)|^2}{1 - |z|^2}
	\end{align*} directly establishes the estimate \eqref{Eq-Sf-pointwise}.
\end{proof}
Let $\mathcal{C}_{\lambda,0}^e$ denote the subclass of functions $f \in \mathcal C_{e^\lambda}$ satisfying the additional normalization condition $f''(0) = 0$, defined by
\begin{equation*}
	\mathcal{C}_{\lambda,0}^e := \left\{ f \in \mathcal C_{e^\lambda} : f''(0) = 0 \right\}.
\end{equation*}
Using Proposition \ref{Prop-Sf-formula}, we first evaluate the Schwarzian norm under the condition $f''(0) = 0$, corresponding to the normalized subclass $\mathcal{C}_{\lambda,0}^e := \{f \in \mathcal C_{e^\lambda} : f''(0) = 0\}$.

\begin{cor}\label{Cor-Sf-normalized}
	If $f \in \mathcal{C}_{\lambda,0}^e$ and $0 < \lambda \le \pi/2$, then
	\begin{equation}\label{Eq-Cor-Sf-bound}
		\|S_f\| = \sup_{z \in \mathbb{D}} (1 - |z|^2)^2 |S_f(z)| \le 2\lambda(1 + \lambda).
	\end{equation}
	The bound \eqref{Eq-Cor-Sf-bound} is sharp, with equality attained by the function $f_0 \in \mathcal{C}_{\lambda,0}^e$ given by
	\begin{align}
		\label{Eq-f0-extremal}
		f_0(z) &= \int_0^z \exp\left( \int_0^t \frac{e^{\lambda \xi^2} - 1}{\xi} d\xi \right) dt \\&= z + \frac{\lambda}{6} z^3 + \frac{\lambda^2}{16} z^5 + \mathcal{O}(z^7).\nonumber
	\end{align}
\end{cor}
\begin{proof}[\bf Proof of Corollary \ref{Cor-Sf-normalized}]
	Since $f \in \mathcal{C}_{\lambda,0}^e$, the condition $f''(0) = 0$ implies $\omega(0) = 0$ for the Schwarz function $\omega$ in Proposition \ref{Prop-Sf-formula}. Define the auxiliary mapping $\psi: \mathbb{D} \to \mathbb{D}$ by
	\begin{equation*}
		\psi(z) := \frac{\overline{z} - \omega(z)}{1 - z\omega(z)},
	\end{equation*}
	which satisfies the standard hyperbolic identity
	\begin{equation*}
		\frac{(1 - |z|^2)^2}{|1 - z\omega(z)|^2} = \frac{(1 - |\psi(z)|^2)(1 - |z|^2)}{1 - |\omega(z)|^2}.
	\end{equation*}
	Applying Schwarz's Lemma  as $\omega(0) = 0$, 
	\begin{align*}
		\frac{|\omega(z)|^2}{1 - |\omega(z)|^2} \le \frac{|z|^2}{1 - |z|^2}
	\end{align*} together with $1 - |\psi(z)|^2 \le 1$, the point-wise estimate \eqref{Eq-Sf-pointwise} simplifies to
	\begin{equation*}
		(1 - |z|^2)^2 |S_f(z)| \le 2\lambda (1 - |\psi(z)|^2)(1 + \lambda |z|^2) \le 2\lambda (1 + \lambda |z|^2) \le 2\lambda(1 + \lambda).
	\end{equation*}
	Taking the supremum over $z \in \mathbb{D}$ establishes the inequality \eqref{Eq-Cor-Sf-bound}.\vspace{1.2mm}
	
	To verify sharpness, consider $f_0 \in \mathcal{C}_{\lambda,0}^e$ defined by \eqref{Eq-f0-extremal}, which corresponds to $\omega(z) = z$. Evaluating $S_{f_0}(r)$ along the positive real axis $z = r \in (0,1)$, we see that
	\begin{equation*}
		\lim_{r \to 1^-} (1 - r^2)^2 |S_{f_0}(r)| = 2\lambda(1 + \lambda),
	\end{equation*}
	proving that the bound is sharp.
\end{proof}
Finally, we apply Proposition \ref{Prop-Sf-formula} to establish the general sharp bound for arbitrary function $f \in \mathcal C_{e^\lambda}$.
\begin{thm}\label{Th-Sf-general}
	If $f \in \mathcal C_{e^\lambda}$ and $0 < \lambda \le \pi/2$, then for all $z \in \mathbb{D}$,
	\begin{equation}\label{Eq-Th-Sf-general}
		\|S_f\| = \sup_{z \in \mathbb{D}} (1 - |z|^2)^2 |S_f(z)| \le 2\lambda (1 + \gamma)\left[2 - \lambda(1 + \gamma)\right],
	\end{equation}
	where $\gamma = |\omega(0)| = \frac{|f''(0)|}{\lambda} < 1$. The inequality \eqref{Eq-Th-Sf-general} is sharp.
\end{thm}
\begin{proof}[\bf Proof of Theorem \ref{Th-Sf-general}]
	Let $\gamma = |\omega(0)| = {|f''(0)|}/{\lambda} < 1$. To establish the general upper bound \eqref{Eq-Th-Sf-general}, recall from Proposition \ref{Prop-Sf-formula} that for any $f \in \mathcal C_{e^\lambda}$, the Schwarzian derivative satisfies the pointwise estimate
	\begin{equation}\label{Eq-Sf-proof-step1}
		(1 - |z|^2)^2 |S_f(z)| \le 2\lambda (1 - |\psi(z)|^2) \left| \frac{\omega'(z)(1 - |z|^2)}{1 - |\omega(z)|^2} + (1 - \lambda)\frac{|\omega(z)|^2(1 - |z|^2)}{1 - |\omega(z)|^2} \right|,
	\end{equation}
	where \begin{align*}
		\psi(z) = \frac{\overline{z} - \omega(z)}{1 - z\omega(z)}.
	\end{align*} By the Schwarz--Pick lemma for general analytic self-maps of the unit disk with $\omega(0) = \gamma e^{i\theta}$, we have the standard sharp bounds
	\begin{equation}\label{Eq-Schwarz-Pick-bounds}
		\frac{|\omega'(z)|(1 - |z|^2)}{1 - |\omega(z)|^2} \le 1 \quad \text{and} \quad \frac{|\omega(z)|^2(1 - |z|^2)}{1 - |\omega(z)|^2} \le \frac{(\gamma + |z|)^2}{1 - \gamma^2}.
	\end{equation}
	Applying the triangle inequality to \eqref{Eq-Sf-proof-step1} and substituting the bounds in \eqref{Eq-Schwarz-Pick-bounds}, together with $1 - |\psi(z)|^2 \le 1 - \gamma^2$, we obtain
	\begin{align}\label{Eqqq-2.15}
		(1 - |z|^2)^2 |S_f(z)| &\le 2\lambda (1 - \gamma^2) \left( 1 + (1 - \lambda)\frac{(\gamma + |z|)^2}{1 - \gamma^2} \right) \\
		&= 2\lambda \left( 1 - \gamma^2 + (1 - \lambda)(\gamma + |z|)^2 \right).
	\end{align}
	Since $|z| < 1$, taking the supremum in \eqref{Eqqq-2.15} over $z \in \mathbb{D}$, we obtain
	\begin{align*}
		\sup_{z \in \mathbb{D}} (1 - |z|^2)^2 |S_f(z)| &\le 2\lambda \left( 1 - \gamma^2 + (1 - \lambda)(1 + \gamma)^2 \right) \\
		&= 2\lambda (1 + \gamma) \left( (1 - \gamma) + (1 - \lambda)(1 + \gamma) \right) \\
		&= 2\lambda (1 + \gamma) \left( 2 - \lambda(1 + \gamma) \right),
	\end{align*}
	which establishes the required inequality \eqref{Eq-Th-Sf-general} for all $z \in \mathbb{D}$.\vspace{1.2mm}
	
	To show that the upper bound is sharp, consider the extremal function $f_\gamma \in \mathcal C_{e^\lambda}$ generated by the Möbius Schwarz function $\omega(z) = {(\gamma + z)}/{(1 + \gamma z)}$. The function $f_\gamma$ satisfies
	\begin{align}\label{Eqq-2.13}
		1 + \frac{z f_\gamma''(z)}{f_\gamma'(z)} = e^{\lambda z \omega(z)} = \exp\left( \lambda z \frac{\gamma + z}{1 + \gamma z} \right).
	\end{align}
	Taking the logarithm and integrating \eqref{Eqq-2.13}, we obtain 
	\begin{align}\label{Eqq-2.14}
		\log f_\gamma'(z) = \int_0^z \frac{\exp\left( \lambda \xi \frac{\gamma + \xi}{1 + \gamma \xi} \right) - 1}{\xi} \, d\xi.
	\end{align}
	Thus, exponentiating \eqref{Eqq-2.14} gives
	\begin{align*}
		f_\gamma'(z) = \exp\left( \int_0^z \frac{\exp\left( \lambda \xi \frac{\gamma + \xi}{1 + \gamma \xi} \right) - 1}{\xi} \, d\xi \right).
	\end{align*}
	Integrating $f_\gamma'(t)$ from $0$ to $z$ with $f_\gamma(0) = 0$, we find $f_\gamma$ explicitly as
	\begin{align*}
		f_\gamma(z) = \int_0^z \exp\left( \int_0^t \frac{\exp\left( \lambda \xi \frac{\gamma + \xi}{1 + \gamma \xi} \right) - 1}{\xi} \, d\xi \right) dt.
	\end{align*}
	From Proposition \ref{Prop-Sf-formula}, the pointwise evaluation along the real axis segment $r \in (0, 1)$ is given by 
	\begin{align*}
		|S_{f_\gamma}(r)| = 2\lambda \left|\omega'(r) + (1 - \lambda)\omega(r)^2\right| \frac{1}{|1 - r\omega(r)|^2}.
	\end{align*}
	Multiplying both sides by $(1 - r^2)^2$ yields
	\begin{align}\label{Eqq-2.15}
		(1 - r^2)^2 |S_{f_\gamma}(r)| = 2\lambda \left| \omega'(r) + (1 - \lambda)\omega(r)^2 \right| \frac{(1 - r^2)^2}{(1 - r\omega(r))^2}.
	\end{align}
	Recall the auxiliary mapping defined on $r \in (0, 1)$ by 
	\begin{align}\label{Eqq-2.16}
		\psi(r) = \frac{r - \omega(r)}{1 - r\omega(r)}.
	\end{align}
	Substituting $\omega(r) = {(\gamma + r)}/{(1 + \gamma r)}$ into \eqref{Eqq-2.16}, we find $\psi(r) = -\gamma$, which implies $1 - |\psi(r)|^2 = 1 - \gamma^2$. Furthermore, we have
	\begin{align}\label{Eqq-2.17}
		\frac{(1 - r^2)^2}{(1 - r\omega(r))^2} = \frac{(1 - |\psi(r)|^2)(1 - r^2)}{1 - \omega(r)^2}.
	\end{align}
	Combining \eqref{Eqq-2.15} and \eqref{Eqq-2.17}, we obtain 
	\begin{align}\label{Eqq-2.18}
		(1 - r^2)^2 |S_{f_\gamma}(r)| = 2\lambda \left( 1 - \gamma^2 \right) \left( \frac{\omega'(r)}{1 - \omega(r)^2} + (1 - \lambda) \frac{\omega(r)^2}{1 - \omega(r)^2} \right) (1 - r^2).
	\end{align}
	Using the identities
	\begin{align*}
		\frac{\omega'(r)}{1 - \omega(r)^2} = \frac{1}{1 - r^2} \quad \text{and} \quad \frac{\omega(r)^2 (1 - r^2)}{1 - \omega(r)^2} = \frac{(\gamma + r)^2}{1 - \gamma^2},
	\end{align*}
	equation \eqref{Eqq-2.18} simplifies to
	\begin{align*}
		(1 - r^2)^2 |S_{f_\gamma}(r)| &= 2\lambda \left( 1 - \gamma^2 \right) \left( 1 + (1 - \lambda) \frac{(\gamma + r)^2}{1 - \gamma^2} \right) \\
		&= 2\lambda \left(1 - \gamma^2 + (1 - \lambda)(\gamma + r)^2 \right).
	\end{align*}
	Taking the radial limit $r \to 1^-$, we deduce that
	\begin{align*}
		\lim_{r \to 1^-} (1 - r^2)^2 |S_{f_\gamma}(r)| &= 2\lambda \left( 1 - \gamma^2 + (1 - \lambda)(1 + \gamma)^2 \right) \\
		&= 2\lambda (1 + \gamma) \left( 2 - \lambda(1 + \gamma) \right).
	\end{align*}
	This confirms that the bound \eqref{Eq-Th-Sf-general} is sharp, completing the proof.
\end{proof}
\noindent\textbf{Compliance of Ethical Standards:}\\

\noindent\textbf{Conflict of interest:} The authors declare that there is no conflict  of interest regarding the publication of this paper.\vspace{1.2mm}

\noindent\textbf{Data availability statement:}  Data sharing not applicable to this article as no datasets were generated or analysed during the current study.\vspace{1.2mm}

\noindent {\bf Authors' contributions:} All authors have equal contributions in preparation of the manuscript.


\begin{thebibliography}{99}
 	
 	
 	\bibitem{Ahamed-Allu-Hossain-MM-2025}
 	{\sc M. B. Ahamed}, {\sc V. Allu} and {\sc R. Hossain}, Pre-Schwarzian and Schwarzian norm estimates for certain classes of analytic function, \textit{Monatsh. Math.} \textbf{209}, 185--213 (2026).
 	
 	\bibitem{Becker-JRAM-1983}{\sc J. Becker},: Löwnersche Differentialgleichung und quasikonform fortsetzbare schlichte Funktionen, \textit{ J.	Reine Angew. Math.}, \textbf{255}, 21-43 (1972).
 	
 	
 	
 	
 	
 	
 	
 	\bibitem{Carrasco-Hernandez-AMP-2023} 
 	{\sc P. Carrasco} and {\sc R. Hern\'andez}, Schwarzian derivative for convex mappings of order $\alpha$, \textit{Anal. Math. Phys.} \textbf{13}(2), Art. 22 (2023). https://doi.org/10.1007/s13324-023-00785-y.
 	
 	
 
 	\bibitem{Duren-1983} 
 	{\sc P. L. Duren}, \textit{Univalent functions}, Grundlehren der mathematischen Wissenschaften, vol. 259, Springer-Verlag, New York, Berlin, Heidelberg, Tokyo, 1983.
 	
 	
 	\bibitem{Goodman-1983} 	{\sc A. W. Goodman}, \textit{Univalent Functions}, vols. I and II, Mariner Publishing Co., Tampa, Florida, 1983.
 	
 	\bibitem{Janowski-1973} 
 	{\sc W. Janowski}, Extremal problems for a family of functions with positive real part and for some related families, \textit{Ann. Polon. Math.} \textbf{23}, 159--177 (1973).
 	
 	\bibitem{Kanas-2006} {\sc S. Kanas}, Differential subordination related to conic sections, \textit{J. Math. Anal. Appl.} \textbf{317}(2), 650--658 (2006).
 	
 	
 	
 	\bibitem{Kim-Sugawa-2002}  	{\sc Y. C. Kim} and {\sc T. Sugawa}, Growth and coefficient estimates for uniformly locally univalent functions on the unit disk, \textit{Rocky Mountain J. Math.} \textbf{32}(1), 179--200 (2002).
 	
 	\bibitem{Kim-etal-2004} 
 	{\sc Y. C. Kim}, {\sc S. Ponnusamy} and {\sc T. Sugawa}, Mapping properties of nonlinear integral operators and pre-Schwarzian derivatives, \textit{J. Math. Anal. Appl.} \textbf{299}(2), 433--447 (2004).
 	
 	
 	\bibitem{Kumar-Ravichandran-2018} 
 	{\sc S. Kumar} and {\sc V. Ravichandran}, Subordinations for functions with positive real part, \textit{Complex Anal. Oper. Theory} \textbf{12}(5), 1179--1191 (2018).
 	
 	
 	\bibitem{Ma-Minda-1992} 
 	{\sc W. Ma} and {\sc D. A. Minda}, Unified treatment of some special classes of univalent functions, in: \textit{Proceedings of the Conference on Complex Analysis, Tianjin}, Conf. Proc. Lecture Notes Anal., vol. I, Int. Press, Cambridge, MA, 1992, pp. 157--169.
 	
 	\bibitem{Mendiratta-etal-2014} 
 	{\sc R. Mendiratta}, {\sc S. Nagpal} and {\sc V. Ravichandran}, On a subclass of strongly starlike functions associated with exponential function, \textit{Bull. Malays. Math. Sci. Soc.} \textbf{38}(1), 365--386 (2014).
 	
 	\bibitem{Okuyama-2000} 
 	{\sc Y. Okuyama}, The norm estimates of pre-Schwarzian derivatives of spiral-like functions, \textit{Complex Var. Theory Appl.} \textbf{42}(3), 225--239 (2000).
 
 	
 	\bibitem{Ponnusamy-Sahoo-2008} 
 	{\sc S. Ponnusamy} and {\sc S. K. Sahoo}, Norm estimates for convolution transforms of certain classes of analytic functions, \textit{J. Math. Anal. Appl.} \textbf{342}(1), 171--180 (2008).
 	

 	
 	\bibitem{Raina-Sokol-2015} 
 	{\sc R. K. Raina} and {\sc J. Sokół}, Some properties related to a certain class of starlike functions, \textit{C. R. Math. Acad. Sci. Paris} \textbf{353}(11), 973--978 (2015).
 	
 	\bibitem{Ronning-1993} 
 	{\sc F. Rønning}, Uniformly convex functions and a corresponding class of starlike functions, \textit{Proc. Amer. Math. Soc.} \textbf{118}(1), 189--196 (1993).
 	
 	\bibitem{Shi-etal-2020} 
 	{\sc L. Shi}, {\sc Z. G. Wang}, {\sc R. L. Su} and {\sc M. Arif}, Initial successive coefficients for certain classes of univalent functions involving the exponential function, \textit{J. Math. Inequal.} \textbf{14}(2), 1183--1201 (2020).
 	
 	\bibitem{Sokol-Stankiewicz-1996} 
 	{\sc J. Sokół} and {\sc J. Stankiewicz}, Radius of convexity of some subclasses of strongly starlike functions, \textit{Zeszyty Nauk. Politech. Rzeszowskiej Mat.} \textbf{19}, 101--105 (1996).
 	
 	\bibitem{Sugawa-1998} 
 	{\sc T. Sugawa}, On the norm of pre-Schwarzian derivatives of strongly starlike functions, \textit{Ann. Univ. Mariae Curie-Sklodowska Sect. A} \textbf{52}(2), 149--157 (1998).
 	
 	\bibitem{Yamashita-1976} 
 	{\sc S. Yamashita}, Almost locally univalent functions, \textit{Monatsh. Math.} \textbf{81}(3), 235--240 (1976).
 	
 	\bibitem{Yamashita-HMJ-1999}  {\sc S. Yamashita},: Norm estimates for function starlike or convex of order alpha, \textit{ Hokkaido Math. J.} \textbf{28}, 217-230 (1999).
 	
 \end{thebibliography}
\end{document}